\documentclass[12pt]{elsarticle}
\usepackage{float}
\usepackage[utf8]{inputenc}
\usepackage{upgreek}
\usepackage{amssymb}
\usepackage{graphicx}
\usepackage{epstopdf}
\usepackage{latexsym}
\usepackage{amsmath}
\usepackage{amssymb}
\usepackage{latexsym}
\usepackage{amsmath}
\usepackage{amsfonts}
\usepackage{amsfonts}
\usepackage{array}
\usepackage{microtype}
\usepackage{amsthm}
\usepackage{hyperref}
\newtheorem{definition}{\bf Definition}[section]
\newtheorem{lem}[definition]{\bf Lemma}
\newtheorem{thm}[definition]{\bf Theorem}
\newtheorem{prop}[definition]{\bf Proposition}
\newtheorem{cor}[definition]{\bf Corollary}

\newtheorem{rem}[definition]{\bf Remark}

\usepackage{hyperref}
\usepackage{xcolor}

\usepackage[symbol]{footmisc}
\def\correspondingauthor{\footnote{*Corresponding author.}}

\begin{document}

\begin{frontmatter}


\title{Classical theorems through convergence methods.}


\author{Fernando Le\'on-Saavedra}
\address{ Department of Mathematics
University of C\'adiz Avda. de la Universidad s/n %11402-Jerez de
la Frontera. Spain.} \ead{fernando.leon@uca.es}

\author{M. del Carmen List\'an-Garc\'{\i}a}
\address{Department of Mathematics
University of C\'adiz, CASEM, Pol. R\'{\i}o San Pedro s/n, 11510-Puerto Real. Spain.} \ead{mariadelcarmen.listan@uca.es}

\author{Mar\'{\i}a Pilar Romero de la Rosa\correspondingauthor{*}}

\address{ Department of Mathematics
University of C\'adiz Avda. de la Universidad s/n 11403-Jerez de
la Frontera. Spain.}
\ead{pilar.romero@uca.es}

\begin{abstract}
We survey some results that provide different versions of classical results through different summability methods. Specifically, in order to adapt such classical results, we analyze which properties should satisfy the summability methods.
Sometimes very sharp conditions are obtained, giving a focused view of the subject and from which new problems emerge.
\end{abstract}


\end{frontmatter}

\section{Introduction}

A good source of problems in Summability Theory consist in considering a classical result, in which the usual convergence is involved, and trying and obtaining new results using the different convergence methods. A step further consists in characterizing those summability methods for which such classical results remain true.

In this note, we follow quite closely some aspects of our research project. This means that this work is by no means an update or comprehensive report on summability methods (a very active field of research with many contributors). We focus instead on some new questions which motivated us on this subject \cite{nuevo,korovkin, orlicz,schur}.

These notes should be seen as an invitation to further reading and a preparation for questions which remain open.
What we tried to do is to present some results, to describe some of the tools and to avoid the technicalities each time it was possible. 

The study of convergence (or summability) methods arises with the development of Fourier Analysis. 
Since then, convergence methods have ge\-ne\-ra\-ted so much interest in Approximation Theory and Applied Mathematics that different monographs have appeared in the literature \cite{mursaleenbook, swartzmath, zygmund2}; moreover, this is a very active field of research.

As usual the symbol $\mathbb{N}$ will denote the set of natural numbers. Let $X$ be a  Banach space.
A linear summability method in $X$ (or convergence method) will be denoted by $\mathcal{R}$; that is, ${\mathcal R}$ will be a linear map ${\mathcal R} : \mathcal{D}_{\mathcal R} \subset X^{\mathbb{N}}\longrightarrow X $ (here ${\mathcal D}_{{\mathcal R}}$ denotes the domain of ${\mathcal R}$). Thus, a sequence $(x_n)\in X^{\mathbb{N}}$ is said to be ${\mathcal R}$-convergent to $L$ (and it will be denoted by $x_n\overset{\mathcal{R}}{\longrightarrow} L$ ) provided ${\mathcal R}((x_n)_{n\geq 1})=L$.
Usually, to avoid bizarre situations, we will require on $\mathcal{R}$ that the limit assignment does not depend on the first terms, that is, for any  $(x_n)_{n\geq 1}\in \mathcal{D}_{\mathcal{R}}$ such that $\mathcal{R}((x_n)_{n\geq 1})=L$ and for any $n_0\in \mathbb{N}$, we have that $(x_n)_{n\geq n_0}\in \mathcal{D}_{\mathcal{R}}$ and 
$\mathcal{R}((x_n)_{n\geq n_0})=L$.

The paper is structured as follows.  In section \ref{seccion2} we will recall the classical Orlicz-Pettis Theorem and we will establish an Orlicz-Pettis Theorem for any regular summability method. We will also see that the condition on regularity cannot be dropped.

One of the first versions of Korovkin approximation result through convergence methods is due to Gadjiev and Orhan  \cite{orhan} who obtained a remarkable  version  in terms of the statistical convergence.  Since then, this research line has been extended  for different types of convergence methods (\cite{korov, b3,EMN,MA, MVEG}). In Section \ref{seccion3} 
 we analyze those summability methods for which Korovkin’s statements continue being true.  Specifically we will show  two properties for summability methods that provide a Korovkin-type approximation result, namely, when the summability method {\it preserves inequalities} or when it {\it preserves order-inequalities}. 
Both concepts fit as a glove to different summability methods and it allows us to obtain several applications.
These  properties  are reminiscent of the squeeze theorem for sequences, as far as we know they have not been  studied for general summability methods and they deserve subsequent studies.

Section \ref{seccion4} deals with the Schur lemma, another classical theorem which was obtained by  Swartz \cite{swartzmath}  for bounded multiplier convergent series. 
We aim to unify different versions of Swartz's result which was obtained incorporating summability methods (\cite{aizpuruarmariopachecoperez,aizpurujavier,aizpurupachecoeslava}, etc). For instance, Schur type results were obtained incorporating any regular matrix summability method (\cite{aizpurupachecoeslava}) or the Banach-Lorentz convergence (\cite{aizpuruarmariopachecoperez}).  The paper concludes with a brief section with concluding remarks and open questions.

\section{Orlicz-Pettis type theorems}
\label{seccion2}
The Orlicz-Pettis Theorem is a classic result concerning  convergent  series, so beautiful that has attracted the interest of many mathematicians and it has been strengthened and generalized in many directions. An early survey is Kalton's paper \cite{kalton}. The reader can see in \cite{a1,a2,a3,a4,ornuevo,a5,a6,a7},  recent results about the Orlicz-Pettis type Theorems.

Let us recall that a series $\sum_i x_i$ in a Banach space $X$ is said to be unconditionally convergent (u.c) if for each permutation of the natural numbers $\pi \,:\mathbb{N}\to \mathbb{N}$ we have that $\sum_i x_{\pi(i)}$ is convergent.
A series $\sum_i x_i$ is subseries convergent if for any $M\subset \mathbb{N}$ 
there exists $x_M\in X$ such that the partial sums $S_n^M=\sum_{i=1}^n\chi_M(i)x_i$ converge  to $x_M$ (here $\chi_M(\cdot)$ denotes the characteristic function on $M$).  Both notions, unconditionally convergence and  subseries convergence, are equivalent in Banach spaces.

On the other hand, a series $\sum_i x_i$ is weak-subseries convergent if for any $M\subset \mathbb{N}$ there exists $x_M\in X$ such that the partial sums $S_n^M=\sum_{i=1}^n\chi_M(i)x_i$ converge weakly to $x_M$.  The classical Orlicz-Pettis Theorem states that a series $\sum_ix_i$  is unconditionally convergent if and only if $\sum_{i} x_i$ is weakly subseries convergent.

A first approach to contextualize the Orlicz-Pettis Theorem for any method of general summability, is to introduce the concept {\it summability method induced by a  summability method in $\mathbb{R}$}. Basically, what we do is to use projections to define a  weak summability method. Indeed, let $\rho$ be a linear summability method in $\mathbb{R}$, that is, a subset $D_{\rho}\subset \mathbb{R}^{\mathbb{N}}$ and a linear function $\rho\,: D_{\rho}\to \mathbb{R}$, which assigns a unique real number $\rho((x_n))$ to a sequence $(x_n)\in D_\rho$. 
A summability method $\rho$ is said to be {\it regular}, if for every convergent sequence $\lim_n x_n=x_0$, the sequence $(x_n)$, $\rho$-converges to the same limit.

A summability method $\rho$ induces a weak summability  method $\mathcal{R}$ in any normed space $X$ as follows: A sequence $(x_n)\in X^{\mathbb{N}}$ is $\mathcal{R}$-convergent to $x_0\in X$ if and only if $f(x_n)$ is $\rho$-convergent to $f(x_0)$ for all $f\in X^*$. Let us observe that in general the convergence method $\mathcal{R}$ could be degenerate, that is, $D_{\mathcal{R}}=\emptyset$.
However if $\rho$ is regular then $D_{\mathcal{R}}$ is non empty, moreover $\mathcal{R}$ is also regular.
\begin{prop}
If $\rho$ is regular, then $\mathcal{R}$ is regular.
\end{prop}
\begin{proof}
Let us suppose that $\lim_{n\to \infty}\|x_n-x_0\|=0$, then for each $f\in X^*$, $\lim_{n\to \infty}|f(x_n)-f(x_0)|=0$, since $\rho$ is regular $f(x_n)\overset{\rho}{\longrightarrow} f(x_0)$. Therefore $x_n\overset{\mathcal{R}}{\longrightarrow} x_0$ as desired.
\end{proof}
The following result unify the results in \cite{a1,a2,a3,a4,a5,a6,a7,ornuevo,X,Y,Z,T}.
\begin{thm}
\label{teormain}
Let $X$ be a real Banach space, $\rho$ a  regular summability method on $\mathbb{R}$ and $\mathcal{R}$ the summability method induced by $\rho$. Then a series $\sum_ix_i$ is unconditionally convergent if and only $\sum_ix_i$ is $\mathcal{R}$-subseries convergent in $X$.
\end{thm}
\begin{proof}
Let $\sum_i x_i$ be an unconditionally convergent series, and let $M\subset \mathbb{N}$. By applying the classical Orlicz-Pettis Theorem, we obtain that there exists $x_M\in X$ such that the sequence $S_n^M=\sum_{i=1}^n \chi_M(i)x_i$ weakly converges to $x_M$, that is, for each $f\in X^*$ the sequence $\sum_{i=1}^n \chi_M(i)f(x_i) \overset{|\cdot|}{\longrightarrow} f(x_M)$. Since $\rho$ is regular, we have that, $\sum_{i=1}^n \chi_M(i)f(x_i) \overset{\rho}{\longrightarrow} f(x_M)$ for each $f\in X^*$, that is $S_n^M=\sum_{i=1}^n \chi_M(i)x_i$ $\mathcal{R}$-converges to $x_M$, as desired.

Assume that for any $M\subset \mathbb{N}$, there exists $x_M$ such that the partial sums $S_n^M=\sum_{i=1}^n \chi_M(i)x_i$ $\mathcal{R}$-converges to $x_M$.
First of all,  we will prove that $\sum_i x_i$ is a  weakly unconditionally Cauchy series. If not, let us argue by contradiction, so, let us suppose that there exists $f\in X^*$ such that $\sum_i |f(x_i)|=+\infty$. Let us consider the following subsets $M=\{i\in \mathbb{N}\,:\,f(x_i)\geq 0\}$ and $N=\{i\in \mathbb{N}\,:\, f(x_i)<0\}$. And let us define the sequence 
$$
\varepsilon_i=\begin{cases} 1 & \mbox{if  } i\in M\\
-1 & \mbox{if  } i\in N,
\end{cases}
$$
then $\sum_{i=1}^\infty \varepsilon_i f(x_i)=+\infty$, hence the sequence $\sum_{i=1}^n\varepsilon_ix_i$ does not $\mathcal{R}$-converge to any $L\in X$. On the other hand,  by hypothesis, given $M,N\subset \mathbb{N}$ there exist $x_M, x_N\in X$ such that $\sum_{i=1}^n\chi_M (i)x_i \overset{\mathcal{R}}{\longrightarrow} x_M$ and $\sum_{i=1}^n\chi_N (i)x_i \overset{\mathcal{R}}{\longrightarrow} x_N$. Therefore,
$$
\sum_{i=1}^n \varepsilon_i f(x_i)=\sum_{i=1}^n\chi_M(i)f(x_i)-\sum_{i=1}^n\chi_N(i) f(x_i)\overset{\rho}{\longrightarrow} f(x_M)-f(x_N)=f(x_M-x_N),
$$
a contradiction. Therefore, for any $f\in X^*$ we have  $\sum_i|f(x_i)|<\infty$.

Now, let us show that given $M\subset \mathbb{N}$ there exists $x_M\in X$
such that $\sum_{i=1}^n\chi_M(i)x_i$ weakly converges to $x_M$. Let $f\in X^*$, since $\sum_i|f(x_i)|<\infty$ we deduce that the series $\sum_{i=1}^n\chi_M(i)f(x_i)$ is convergent to some $\lambda_{M,f}\in\mathbb{R}$, and hence $\rho$-convergent to $\lambda_{M,f}$. On the other hand, by hypothesis, there exists $x_M\in X$ such that 
$\sum_{i=1}^n\chi_M(i)x_i\overset{\mathcal{R}}{\longrightarrow} x_M$, that is, for each $f\in X^*$ we have that, $\sum_{i=1}^n\chi_M(i)f(x_i)\overset{\rho}{\longrightarrow} f(x_M)$. Therefore $\lambda_{M,f}=f(X_M)$. Hence, we obtain that for any $f\in X^*$ the sequence $\sum_{i=1}^n\chi_M(i)f(x_i)$ converges to $f(x_M)$, that is, $\sum_{i=1}^n\chi_M(i)x_i$ weakly converges to $x_M$. Thus, by applying the classical Orlicz-Pettis Theorem we obtain that the series $\sum_ix_i$ is unconditionally convergent, as desired.
\end{proof}
Let us see that regularity cannot be dropped in the hypothesis of the above theorem.
\begin{rem}
\label{referee}
  Let us consider the following linear summability method, $\rho$, a sequence $(x_n)\in \mathbb{R}^{\mathbb{N}}$ is said to be $\rho$-convergent to $x_0$ if $\lim_{n\to \infty} \frac{x_n}{n^2}=x_0$. Then, clearly in the realm of bounded sequences $\rho_{|\ell^\infty}=0$. Thus $\rho$ is not regular. Now, let us consider $\mathcal{R}$ the summability method  on $\ell^2$  induced by $\rho$. 
  As usual, let us denote by $\{e_i\}_{i\in\mathbb{N}}$ the canonical basis of $\ell_2$
  For every $M\subset \mathbb{N}$ we have that 
$\sum_{i=1}^n \chi_M(i)e_i\overset{\mathcal{R}}{\longrightarrow}0=x_M $. However, $\sum_{i=1}^n \chi_M(i)e_i$ is not norm convergent to $0$. The argument of the proof breaks down if we cannot guarantee that $\lambda_{M,f}=f(x_M)$. This fact highlights the importance of regularity in the proof of the above result. 
\end{rem}

We will show the utility of the Theorem \ref{teormain}, seeing some applications.
We will say that $\mathcal{I} \subset \mathcal{P}(\mathbb{N})$ is a non trivial ideal if 

\begin{enumerate}
    \item $\mathcal{I}\neq \emptyset$ and $\mathcal{I}\neq \mathcal{P}(\mathbb{N})$. 
    \item If $A,B\in\mathcal{I}$ then $A\cup B\in\mathcal{I}$.
    \item If $A\subset B$ and $B\in \mathcal{I}$ then $A\in \mathcal{I}$. 
    \item Additionally we say that $\mathcal{I}$ is  regular (or admissible) if it  contains all finite subsets.
\end{enumerate}

A non-trivial regular ideal $\mathcal{I}$ defines a regular summability method on any metric space.
We will say that a sequence $(x_n)\subset  \mathbb{R}$ is $\mathcal{I}$-convergent to $L\in \mathbb{R}$ (in short $L=\mathcal{I}-\lim_{n\to\infty}x_n$) if for any $\varepsilon>0$ the subset 
$$A(\varepsilon)=\{n\in \mathbb{N}\,: \, |x_n-L|>\varepsilon\}\in \mathcal{I}.$$
Thus, given a Banach space $X$ the $\mathcal{I}$-convergence defines a weakly summability method in $X$, we will say that a sequence $(x_n)\subset X$ is weakly-$\mathcal{I}$ convergent to $x_0$ if for any $f\in X^*$, we have $f(x_n)\overset{\mathcal{I}}{\longrightarrow} f(x_0)$.

\begin{cor}
Let $\mathcal{I}$ be a non-trivial ideal. Then a series $\sum_ix_i$ in a real Banach space $X$ is unconditionally convergent if and only if  $\sum_ix_i$ is subseries weakly-$\mathcal{I}$ convergent. 
\end{cor}
In particular, if we consider the ideal $\mathcal{I}_d$ of all subsets in $\mathbb{N}$ with zero density, the ideal convergence induced by $\mathcal{I}_d$ (which is non trivial and regular) is the statistical convergence. Therefore the above Corollary is also true for the weak-statistical convergence (\cite{a1}).

Now, let us consider a regular matrix summability method induced by an infinite matrix $A=(a_{ij})$, which is defined as follows. A sequence $(x_n)\in \mathbb{R}^{\mathbb{N}}$ is $A$-summable to $L$ if $\lim_{n\to \infty}\sum_{j=1}^\infty a_{nj} x_j=L$. A matrix $A$ is regular if the usual convergence implies the $A$-convergence, and the limits are preserved. 
Now, if $X$ is a Banach space then the matrix $A$ induces also a summability method on $X$, we say that a sequence $(x_n)\in X^{\mathbb{N}}$ is $A$-convergent to $x_0\in X$ if $\lim_{n\to \infty}\sum_{j=1}^\infty a_{nj} x_j=x_0$. The matrix $A$
also induces a weak convergence, a sequence $(x_n)\in X^{\mathbb{N}}$ is weakly $A$-convergent to $x_0\in X$ if for any $f\in X^*$ we have that $f(x_n)$ is $A$-convergent to $f(x_0)$. Applying Theorem \ref{teormain} we get:

\begin{cor}
Let $A$ be a regular matrix. Then a series $\sum_ix_i$ is unconditionally convergent if and only if $\sum_ix_i$ is subseries weak-$A$-convergent.
\end{cor}

Thus we obtain the results in \cite{u1}. In particular if $A$ is the Cesàro matrix, we obtain the results in \cite{AizpuruDavila2006}.

Let us consider the following summability method $\rho$: we will say that a sequence $(x_n)$ is Cesàro, statistically convergent to $L$, if the sequence of Cesàro means $\left(x_1, \frac{x_1+x_2}{2}, \frac{x_1+x_2+x_3}{3}\cdots\right)$ is statistically convergent to $L$. 
Given a Banach space $X$, the summability method $\rho$ induces a weakly summability method in $X$. Namely, we say that a sequence $(x_n)\in X^{\mathbb{N}}$ is weakly statistically Cesàro convergent to $x_0\in X$ if for every $f\in X^*$ the sequence $f(x_n)$ is Cesàro statistically convergent to $f(x_0)$.

As a consequence we obtain the results in \cite{a2}.

\begin{cor}
Let $A$ be a regular matrix. Then a series $\sum_ix_i$ is unconditionally convergent if and only if $\sum_ix_i$ is subseries weakly statistically Cesàro-convergent.
\end{cor}

The proof of  the general Theorem \ref{teormain} is simpler than particular results for a concrete summability method; actually it simplifies the existing proofs. Moreover, it can be widely applied to several summability methods obtaining new versions of the Orlicz-Pettis theorem. For  instance, it can be applied to the Erd\"os-Ulam convergence, $w^p$-Cesàro convergence, $f$-statistical convergence, etc.  (see \cite{aizpurulistanrambla,ornuevo}).

\section{Korovkin-type approximation theorems}
\label{seccion3}

Let  ${\mathcal R}$ be a summability method on  a Banach space $(X,\|\cdot\|)$. We turn our attention to the following properties of a summability method.
\begin{definition}
\label{des1}
Let $\mathcal{R}\,:\, \mathcal{D}_{\mathcal{R}}\subset X^{\mathbb{N}}\to X$ be a summability method. Let us suppose that  the sequences $(w_n),(x_n),(y_n),(z_n)\in \mathcal{D}_{\mathcal{R}} $ satisfy
\begin{equation}
    \label{des}
\|w_n-w\|\leq C(\|x_n-x\|+\|y_n-y\|+\|z_n-z\|)
\end{equation}
 for some $w,x,y,z\in X$, some constant $C>0$ and for all $n\geq 1$. 
We will say that $\mathcal{R}$ {\sf preserves inequalities} if for any sequences $(x_n),(y_n),(z_n)$ satisfying (\ref{des}) and satisfying 
 $x_n\overset{\mathcal{R}}{\longrightarrow} x$, $y_n\overset{\mathcal{R}}{\longrightarrow} y$,  $z_n\overset{\mathcal{R}}{\longrightarrow} z$, we have that $w_n\overset{\mathcal{R}}{\longrightarrow} w$.
\end{definition}

The following notion will be also useful.  Let $(E,<)$ be a Banach lattice endowed with a lattice norm  $\|\cdot\|$. As usual, we denote $|x|:=x\vee -x$.
\begin{definition}
\label{des2}
Assume that $(E,\leq)$ is a Banach lattice endowed with a lattice norm $\|\cdot\|$. Let $\mathcal{R}\,:\,\mathcal{D}_{\mathcal{R}}\subset X^{\mathbb{N}}\to X$ be a summability method. Assume that there exist sequences $w_n,x_n,y_n,z_n\in \mathcal{D}_{\mathcal{R}}$ satisfying
\begin{align}
\label{desi}
-C[(x_n-x)&+(y_n-y)+(z_n-z)]<w_n-w< \nonumber\\
&< C[(x_n-x)+(y_n-y)+(z_n-z)]
\end{align}
for some constant $C>0$. We say that the summability method $\mathcal{R}$  {\sf  preserves order-inequalities} if for any sequences $(x_n),(y_n),(z_n)$ satisfying (\ref{des1}) and satisfying 
 $x_n\overset{\mathcal{R}}{\longrightarrow} x$, $y_n\overset{\mathcal{R}}{\longrightarrow} y$,  $z_n\overset{\mathcal{R}}{\longrightarrow} z$, we have that $w_n\overset{\mathcal{R}}{\longrightarrow} w$.
\end{definition}
Definitions \ref{des2} and \ref{desi} allow us to obtain a general view of Korovkin's approximation results appeared in the literature \cite{korov, b3,EMN,MA, MVEG}. The above definitions and the next theorem can be found in \cite{korovkin}.

\begin{thm}
Let  $\mathcal{R}$ be a summability method on the Banach lattice of continuous functions $\mathcal{C}[0,1]$ endowed with the supremum norm. Assume that $\mathcal{R}$ preserves inequalities or order-inequalities. If $(L_n)$ is a sequence of positive linear operators  from $\mathcal{C}([0,1])$ into $\mathcal{C}([0,1])$ then  for any $f\in\mathcal{C}([0,1])$ and bounded on $\mathbb{R}$, $L_nf\overset{\mathcal{R}}{\longrightarrow}f$ if and only if $L_n 1\overset{\mathcal{R}}{\longrightarrow}1$, $L_n t\overset{\mathcal{R}}{\longrightarrow}t$ and $L_n t^2\overset{\mathcal{R}}{\longrightarrow}t^2$.
\end{thm}

\begin{rem}
Let us consider the classical Bernstein polynomials defined by
$$B_nf(x)=\sum_{k=0}^n f\left(\frac{n}{k}\right) \binom{k}{n} x^k (1-x)^{n-k}\quad 0\leq x\leq 1.$$
Denote by $\mathcal{P}[0,1]$ the set of polynomials oon $[0,1]$ and by
 by $\bigvee \{1,t,t^2\}$ the subspace generated by $1,t$ and $t^2$. Denote by  $\mathcal{T} \,:\,\mathcal{P}[0,1]\rightarrow \bigvee \{1,t,t^2\}$ the standard  projection and
let us consider the following summability method defined on $\mathcal{P}[0,1]\subset \mathcal{C}[0,1]$ the space of polynomials on $[0,1]$ as follows:
$$
p_n\overset{\mathcal{R}}{\longrightarrow} f
$$
if and only if $\mathcal{T} p_n\to f$ uniformly on $[0,1]$ provided such limit exists.
Now let us observe that $(B_n1)(x)=1, (B_nt)(x)=x$ and $(B_nt^2)(x)=x^2+\frac{x-x^2}{n}$ which implies that
$B_n1\overset{\mathcal{R}}{\longrightarrow} 1$, $B_nt\overset{\mathcal{R}}{\longrightarrow} t$ and $B_nt^2\overset{\mathcal{R}}{\longrightarrow} t^2$.
However, by construction we cannot assert that $B_np\overset{\mathcal{R}}{\longrightarrow} p$ for any polynomial $p$.
The argument in the proof breaks down if we can't guarantee that $\mathcal{R}$ preserves inequalities. This fact highlights the importance of this property in the proof of the above result. 
\end{rem}

Thus, in order to obtain Korovkin-type approximation results
 we  need to show only that the summability method $\mathcal{R}$ preserves inequalities or preserves order-inequalities.
 Denote by $\mathcal{C}_{2\pi}(\mathbb{R})$ the continuous and $2\pi$ periodic functions on $\mathbb{R}$. Next, we will show that the ideal convergence preserves inequalities. 
 \begin{thm}
\label{ideal}
Let $\mathcal{I}\subset \mathcal{P}(\mathbb{N})$ be a non-trivial ideal and let us consider the $\mathcal{I}$-convergence on $\mathcal{C}[0,1]$ and on $\mathcal{C}_{2\pi}(\mathbb{R})$. Then, the $\mathcal{I}$-convergence preserves inequalities.
\end{thm}
\begin{proof}
It is sufficient to show the result for the $\mathcal{I}$-convergence on $\mathcal{C}[0,1]$, the same argument is true on $\mathcal{C}_{2\pi}(\mathbb{R})$.
Indeed, assume that there exists $C>0$ such that
\begin{equation}
    \label{ecuacion}
\|w_n-w\|_\infty\leq C(\|x_n-x\|_\infty+\|y_n-y\|_\infty+\|z_n-z\|_\infty)
\end{equation}
for all $n$,  $x_n\overset{\mathcal{I}}{\longrightarrow}x$, $y_n\overset{\mathcal{I}}{\longrightarrow}y$ and $z_n\overset{\mathcal{I}}{\longrightarrow}z$. We wish to show that $w_n\overset{\mathcal{I}}{\longrightarrow}w$.

Fix $\varepsilon>0$, we wish to show that
$$
A(\varepsilon)=\{n\in\mathbb{N}\,:\, \|w_n-w\|_{\infty}>\varepsilon\}\in\mathcal{I}.
$$
By hypothesis;
$A_1=\{n\in\mathbb{N}\,:\, \|x_n-x\|>\frac{\varepsilon}{3C}\}\in \mathcal{I}$, $A_2\{n\in\mathbb{N}\,:\, \|y_n-y\|>\frac{\varepsilon}{3C}\}\in \mathcal{I}$ and
$A_3=\{n\in\mathbb{N}\,:\, \|z_n-y\|>\frac{\varepsilon}{3C}\}\in \mathcal{I}$.

According to (\ref{ecuacion}) we have that $A(\varepsilon)\subset A_1\cup A_2\cup A_3$, then by applying (2) and (3) in the definition of ideal, we  get that $A(\varepsilon)\in\mathcal{I}$ as  we desired.\end{proof}

 \begin{cor}
 Korovkin's statement continues being true if we incorporate the  convergence by a non-trivial ideal.
 \end{cor}
Next we analyse some summability methods that preserve order inequalities, thus Korovkin statements continue being true if we incorporate these summability methods.

A sequence $(f_k)$ in a Banach lattice $(X,<)$ endowed with the lattice norm $\|\cdot\|$ is said to be almost convergent to $L\in X$ if the double sequence
$$
\frac{1}{m+1}\sum_{i=0}^mx_{n+i}
$$
converges to $L$ as $m\to\infty$ uniformly in $n$.

\begin{thm}
Let $(X,<)$ be a Banach lattice with the lattice norm $\|\cdot\|$. Let us consider $\mathcal{R}$ the almost summability defined on $\mathcal{D}_{\mathcal{R}}\subset X^{\mathbb{N}}$. Then $\mathcal{R}$ preserves order-inequalities.
\end{thm}
\begin{proof}
Indeed, assume that there exist sequences $(w_i),(x_i),(y_i)$ and $(z_i)$ satisfying
\begin{align*}
-C\left((x_i-x)+(y_i-y)+(z_i-z)\right) &<w_i-w \\
&< C\left((x_i-x)+(y_i-y)+(z_i-z)\right).
\end{align*}
Thus,
\begin{align*}
\left|\frac{1}{m+1}\sum_{i=0}^mw_{n+i}-w\right|&<C\left|\frac{1}{m+1}\sum_{i=0}^mx_{n+i}-x+\frac{1}{m+1}\sum_{i=0}^my_{n+i}-y\right. \\&+\left.\frac{1}{m+1}\sum_{i=0}^mz_{n+i}-z\right|,
\end{align*}
hence, taking norms
\begin{align*}
\left\|\frac{1}{m+1}\sum_{i=0}^mw_{n+i}-w\right\|&<C\left[\left\|\frac{1}{m+1}\sum_{i=0}^mx_{n+i}-x\right\|+\left\|\frac{1}{m+1}\sum_{i=0}^my_{n+i}-y\right\|\right. \\
&+\left.\left\|\frac{1}{m+1}\sum_{i=0}^mz_{n+i}-z\right\|\right].
\end{align*}
Letting $m\to \infty$ we get
$$
\left\|\frac{1}{m+1}\sum_{i=0}^mw_{n+i}-w\right\|\to 0
$$
uniformly in $n$, therefore $w_n\overset{\mathcal{R}}{\longrightarrow}w$ as we desired.
\end{proof}
\begin{cor}
 Korovkin's statement continues being true if we incorporate the almost summability.
\end{cor}

Let $A=(\alpha_{ij})_{(i,j)\in \mathbb{N}\times\mathbb{N}}$ be a matrix with  non-negative entries.   A sequence $(x_i)$ in a Banach lattice $(X,<)$ is said to be $A$-summable (or $A$-convergent) to $L\in X$, if $\lim_n \sum_j \alpha_{nj} x_j=L$.  
A non-negative matrix $A$ is said to be regular if
\begin{enumerate}
    \item $\sup_n\sum_j\alpha_{nj}<\infty$.
    \item $\lim_n \alpha_{ni}=0$.
    \item $\lim_n\sum_i\alpha_{ni}=1$.
\end{enumerate}

\begin{thm}
Let $(X,<)$ be a Banach lattice with the lattice norm $\|\cdot\|$. Let $A$ be a regular matrix with non-negative entries. The $A$-convergence   preserves order-inequalities.
\end{thm}
\begin{proof}
Indeed, assume that there exist sequences $(w_i),(x_i),(y_i)$ and $(z_i)$ satisfying
\begin{align*}
-C\left((x_i-x)+(y_i-y)+(z_i-z)\right) &<w_i-w< \\
&< C\left((x_i-x)+(y_i-y)+(z_i-z)\right).
\end{align*}
Since $e_n\overset{A}{\longrightarrow}e$, for $e_j=x_j,y_j,z_j$, $e=x,y,z$ and $\lim_n\sum_i\alpha_{ni}=1$ we get:
$$
\left\|\sum_j \alpha_{nj} e_j-\sum_{j}a_{nj}e\right\|\leq 
\left\|\sum_j \alpha_{nj} e_j-e\right\|+\left|1-\sum_{j}a_{nj}\right| \|e\|\to 0
$$
as $n\to\infty$ for $e=x,y,z$.

Since the entries of $A$ are non-negative, for every $n$ we get
\begin{align*}
\left|\sum_j \alpha_{nj} w_j-\left(\sum_j \alpha_{nj}\right)w \right|&\leq C\left| \sum_j \alpha_{nj} x_j-\left(\sum_{j}a_{nj}\right)x\right.\\
&+\sum_j \alpha_{nj} y_j-\left(\sum_{j}a_{nj}\right)y\\
&\left.+\sum_j \alpha_{nj} z_j-\left(\sum_j\alpha_{nj}\right)z\right| .
\end{align*}
Taking norms and applying the triangular inequality we get that
$$
\lim_{n\to \infty }\left\|\sum_j \alpha_{nj} w_j-\left(\sum_{j}\alpha_{nj}\right)w\right\|=0.
$$
Hence, since
$$
\left\|\sum_j \alpha_{nj} w_j-w\right\|\leq 
\left|\sum_j \alpha_{nj} -1\right|\|w\|+\left\|\sum_j \alpha_{nj} w_j-\left(\sum_{j}\alpha_{nj}\right)w\right\|,
$$
and each summand in the right hand  converges in norm to zero as $n$ tends to $\infty$, we get that $w_n\overset{A}{\longrightarrow}w$
as we desired.
\end{proof}
\begin{cor}
 Korovkin's statement remains true if we incorporate a regular matrix summability method with non-negative entries.
\end{cor}

\section{Schur's Lemma for Multiplier series}
\label{seccion4}
Schur's Lemma is a very striking result, so that, it has attracted the interest of many people. One of the classical versions (\cite{dunford}) states that a sequence in $\ell_1$ is weakly convergent if and only if it is norm convergent. A later version  was discovered by Antosik and Swartz using the Basic Matrix Theorem (see \cite{antosikmatrix}), moreover Swartz \cite{swartzmath,booksw} obtained a version of the Schur lemma for bounded multiplier convergent series.

Let   $(X,\|\cdot\|)$ be a real Banach space. Let us denote by $\ell_\infty(X)$ the space of all bounded sequences in $X$ provided with the supremum norm (which we will denote sometimes abusively by $\|\cdot\|$):
 $$
 \|(x_k)\|_{\ell_\infty(X)}=\sup \{\|x_n\|\,\,:\,\,n\in\mathbb{N}\}.
 $$
 
It is well known that a series $\sum_ix_i$ is (wuc) if and only if
$\sum_ia_ix_i$ is convergent for every sequence $(a_i)\in c_0$, or equivalently $\{\sum_{i=1}^na_ix_i\,:\, (a_i)\in B_{\ell_\infty},n\in\mathbb{N}\}$ is bounded in the Banach space $X$. It is also known that a series  $\sum_i x_i$ is (uc) if and only if $\sum_ia_ix_i$ is convergent for every $(a_i)\in \ell_\infty$ (see \cite{diestel}[Chapter V]).

Let us denote by $X(c_0)$ the (wuc) series and  $X(\ell_{\infty})$ will denote the space of all (uc) series. Both spaces are real Banach spaces, endowed with the norm:
\begin{equation}
    \label{norma}
    \|(x_k)_{k\in \mathbb{N}}\|_s=\sup\left\{ \left\| \sum_{i=1}^n a_ix_i\right\|\,\,:\,\,|a_i|\leq 1, i\in\{1,\cdots, n\}, n\in \mathbb{N}   \right\}.
\end{equation}
The following  result  by Swartz  \cite{swartzmath}  is a striking version of the Schur lemma 
for bounded multiplier convergent series:
\begin{thm}[Swartz-1983]
\label{swartz}
Let $(x_n)_{n\in\mathbb{N}}=(x_n(k))$ be a sequence in $X(\ell_\infty)$ such that for every $(a(k))\in \ell_\infty$, $\lim_{n\to\infty}\sum_{k=1}^\infty a(k)x_n(k)$ exists. Then, there exists $x_0\in X(\ell_\infty)$ such that $\lim_{n\to\infty}\|x_n-x_0\|_s=0$.
\end{thm}

More effort is needed to obtain Theorem \ref{swartz} for general summability methods.
Our work extends  Swartz's result in two directions: 1) To know which properties should satisfy a summability method to establish the Schur Lemma for multiplier series, and 2) to extend the result  for subspaces of $\ell_\infty $, containing $c_0$ which are $\ell_\infty$-Grothendieck (see \cite{nuevo} for new results on this property).

Let   $S$ be a closed subspace of $\ell_\infty$ containing $c_0$. Here $\mathcal{R}\,:D_{\mathcal{R}}\subset X^{\mathbb{N}}\to X$  defines a summability method defined on a real Banach space $X$. Let us consider the following vector space:
$$
X(S,\mathcal{R})=\left\{(x_k)_{k\in\mathbb{N}}\,\, :\, \mathcal{R}\left(\left(\sum_{k=1}^n a_kx_k\right)\right) \,\,\textrm{exists for every } (a_k)_{k\in \mathbb{N}}\in S\right\}.
$$

To establish our result we show up some properties on a summability method, that have not been  treated and that deserve subsequent studies.

 \begin{description}
 \item[(h1)]{\bf Regularity.}
     \item[(h2)] {\bf $\mathcal{R}$-weak convergence.} For every $(x_n)\in D_{\mathcal{R}}\cap \ell_\infty(X)$ such that $\mathcal{R}((x_n))=L$  it is satisfied that  $\sup_{f\in B_{X^*}}|f(x_n)-f(L)|=0$.
     \item[(h3)]{\bf Boundedness.} In the following sense, there exists $M>0$ such that $\|\mathcal{R}((x_n))\|\leq M \|(x_n)\|_{\ell_\infty(X)}$ for all $(x_n)\in D_\mathcal{R}\cap \ell_{\infty}(X)$.
     \item[(h4)] $\mathcal{R}$-{\bf completeness}. That is, a sequence $(x_n)\in D_\mathcal{R}\cap \ell_\infty(X)$ if and only if $(x_n)$ is $\mathcal{R}$-Cauchy. That is, for any $\varepsilon>0$ there exists $n_0\in \mathbb{N}$, such that, $(x_n-x_{n_0})_{n\geq n_0}\in D_\mathcal{R}\cap \ell_{\infty}(X)$ and  $\|\mathcal{R}((x_n-x_{n_0})_{ n\geq n_0}))\|<\varepsilon$.
    
 \end{description}

Hypothesis {\bf (h3)} will guarantee that $X(S,\mathcal{R})$ is a closed subspace of $X(c_0)$ endowed with the norm $\|\cdot\|_s$, this is one of our first results. We include the proof to see the use of properties in action.

 \begin{thm}
 \label{main1}
 Let $\mathcal{R}$ be a convergence method on a Banach space $X$ satisfying {\bf (h3)}. Then $X(S,\mathcal{R})$ is a closed subspace of $X(c_0)$ endowed with the norm $\|\cdot\|_s$.
 \end{thm}
\begin{proof}
 Let $(x^n)\in X^{\mathbb{N}}(S,\mathcal{R})$ satisfying $\lim_{n\to\infty} \|x^n-x^0\|_s=0$ for some $x^0=(x_i^0)\in X(c_0)$ and let us show that $x^0\in X(S,\mathcal{R})$, that is, for all $(a_k)\in S$ we have that $\sum_{k=1}^na_kx_k^0$ is $\mathcal{R}$-convergent.

By hypothesis, $\mathcal{R}$ satisfies {\bf (h3)}, therefore there exists $M>0$ such that
$$
\|\mathcal{R}((x_k))-\mathcal{R}((y_k))\|\leq M \|(x_k-y_k)\|_{\ell_{\infty}(X)}.
$$
for all $(x_n),(y_n)\in \ell_{\infty}(X)\cap D_\mathcal{R}$.

Since $(x^n)$ is a Cauchy sequence, for each $\varepsilon>0$ there exists $k_0$, such that, for all $p,q\geq k_0$, $\|x^p-x^q\|_s<\varepsilon/M$.

 Let us fix $(a_k)$ in the unit ball of $S$. 
 Since $x^m\in X(S,\mathcal{R})$, we obtain that 
 the partial sums $\sum_{k=1}^n a_kx_k^m$ are 
 $\mathcal{R}$-convergent to some $y_m\in X$. 
 Then, for $p,q\geq k_0$, by applying {\bf (h3)} we get
\begin{align*}
\|y_p-y_q\|&=\left\|\mathcal{R}\left(\left(\sum_{k=1}^n a_kx_k^q\right)\right)-\mathcal{R}\left(\left(\sum_{k=1}^n a_kx_k^p\right)\right)\right\|\\
%&\leq M \sup_n \left\|\sum_{k=1}^{n}(a_k^p-a_k^q)x_k\right\|
%\\
& \leq M \|x^p-x^p\|_{s}\leq \varepsilon.
\end{align*}
 Thus $(y_m)$ is a Cauchy sequence. Since $X$ is complete, let $y_0$ be its limit. We claim that $\mathcal{R}\left(\sum_{k=1}^n a_kx_k^0\right)=y_0$.
 Indeed, for any $\varepsilon>0$ there exists $p$ such that $\|y_p-y_0\|\leq \frac{\varepsilon}{2}$ and $\|x^p-x^0\|_s\leq \frac{\varepsilon}{2M}$.
Since $\mathcal{R}$ satisfies  {\bf (h3)}:
 $$
 \left\|\mathcal{R}\left(\left(\sum_{k=1}^n a_kx_k^p\right)\right)-\mathcal{R}\left(\left(\sum_{k=1}^na_k x_k^0\right)\right)\right\|\leq M \|x^p-x^0\|_s\leq 
 \frac{\varepsilon}{2}.
 $$
 Hence,
 \begin{align*}
      \left\|\mathcal{R}\left(\left(\sum_{k=1}^n a_kx_k^0\right)\right)-y_0\right\|& =
       \left\|\mathcal{R}\left(\left(\sum_{k=1}^n a_kx_k^0\right)\right)-y_p+y_p-y_0\right\|\\
       &\leq \left\|\mathcal{R}\left(\left(\sum_{k=1}^n a_kx_k^p\right)\right)-\mathcal{R}\left(\left(\sum_{k=1}^na_k x_k^0\right)\right)\right\|+\|y_p-y_0\|\\
       &\leq \frac{\varepsilon}{2}+\frac{\varepsilon}{2}=\varepsilon.
 \end{align*}
 Since $\varepsilon$ was arbitrary, we obtain that $\mathcal{R}\left(\sum_{k=1}^n a_kx_k^0\right)=y_0$ as we desired.
\end{proof}

One of the keys to the proof of Theorem \ref{main} is to ensure certain operators in wich definition  $R$ is involved, are bounded;  we can guarantee this condition thanks  to the hypothesis {\bf (h3)}.

\begin{lem}
\label{lemaprevio}
Let $X$ be a Banach space and let $\mathcal{R}$ be a convergence method satisfying {\bf (h3)}. For each closed subspace $S$, satisfying $c_0\subset S\subset \ell_\infty$ and  $x=(x_k)\in X(S,\mathcal{R})$ the linear operator
$\Sigma_x\,:\, S\longrightarrow X$, defined by
$$\Sigma_x((a_k)_k)=\mathcal{R}\left(\left(\sum_{k=1}^na_kx_k\right)\right),$$
is bounded.
\end{lem}

A  subspace $M$ of the dual $X^{**}$ of a real Banach space $X$ is called a $M$-Grothendieck space if every sequence in $X^*$ which is $\sigma(X^*,X)$ convergent is also $\sigma(X^*, M)$ convergent. In particular, $X$ is said to be Grothendieck if it is $X^{**}$-Grothendieck, that is, 
every weakly-$*$ convergent sequence in the dual space $X^*$ converges with respect to the weak topology of $X^*$. 

There are many summability methods  which satisfy {\bf (h3)} in Theorem \ref{main}. For instance, of course the usual convergence,  the statistical convergence, lacunary statistical convergence, the uniform almost convergence, any regular bounded matrix summability method, etc.

Next we state the main result in \cite{schur}. It is surprising how this result unifies all known results and it is applied to  most summability methods.
\begin{thm}
\label{main}
Let $X$ be a real Banach space, and let $\mathcal{R}$ be a summability method  satisfying {\bf (h2),(h3)}. Let $(x^n)$ be a sequence in $X(c_0)$. Let $S$ be  a closed
subspace of $\ell_\infty$ containing $c_0$ and assume that $S$
  is $\ell_\infty$-Grothendieck. If for each $(a_k)\in S$, the sequence
$y_n=\sum_{k=1}^\infty a_k x_k^n$ $\mathcal{R}$-converges, then $(x^n)$ converges in $X(c_0)$. 
\end{thm}
\begin{rem}
 Let $S$ be a subspace of $\ell_\infty$ containing $c_0$. We consider the inclusion map $\iota\,:c_0\to S$ and the corresponding bidual map which is an isometry from $c_0^{**}=\ell_\infty$ into $S^{**}$. It is natural  to characterize the subspaces $S$, $c_0\subset S\subset \ell_\infty$  with the property $\ell_\infty$-Grothendieck. 
Theorem \ref{main} is true for $S=\ell_\infty$ but it continues being true for every subspace $S\subset \ell_\infty$ which is $\ell_\infty$-Grothendieck.

There are non-trivial subspaces of $\ell_\infty$ which are $\ell_\infty$-Grothendieck. As it was noted in \cite{aizpuruarmariopachecoperez} Remark 4.1, Haydon constructed, using transfinite induction, a Boolean Algebra $\mathcal{F}$ containing the sets $\{\{i\} \,: \, i \in \mathbb{N}\} $ whose corresponding space $ C (\mathcal{F}) $ can be seen as a proper subspace of  $\ell_\infty$ , it contains $c_0 $ and it is also Grothendieck.
We refer to the interested reader to the paper (\cite{nuevo}) where we analyze the property $\ell_\infty$-Grothendieck and we obtain  natural examples of such subspaces of $\ell_\infty$.
\end{rem}

\section{Concluding remarks and some questions}
\label{seccion5}

Regarding the results of Seccion \ref{seccion2}, we know that it is not possible to drop the hypotesis of regularity on Theorem \ref{teormain}. However, we do not know if the hypothesis on regularity is also necessary, these are our first questions:

\begin{quote}
    {\bf Question. }Is the hypothesis of regularity necessary in  Theorem \ref{teormain}?
\end{quote}
More specifically. Let $A$ be a matrix summability method, not necessarily regular
\begin{quote}
    {\bf Question.} Is it possible to obtain a version of Orlicz-Pettis Theorem for the matrix summability induced by $A$?
\end{quote}

Analogously, in the search for new Korovkin type results, it would be interesting to obtain Korovkin’s statement for a summability method that does not preserve inequalities or order-inequalities. That is, the following question arises:

\begin{quote}
{\bf Question. } Let $A$ be a matrix defining a non-necessarily regular summability method. Is it possible a Korovkin-type result incorporating the $A$-summability? 
In general, for which matrices $A$ is   a Korovkin-type result possible?
\end{quote}

Finally, the results of the section \ref{seccion4} suggest that it should be interesting to obtain more information about $\ell_\infty$-Grothendieck spaces. 

\begin{quote}
{\bf Question. } Is it possible to find an example of and $\ell_\infty$-Grothendieck subspace which is not a Grothendieck space? 
\end{quote}

To study the above question , we suggest that we would need a good characterization of $\ell_\infty$-Grothendieck subspaces, which we do not have yet.

% ------------------------------------------------------------------------
\end{document}